\documentclass[prepint]{RMMCART}
\usepackage{amssymb, amscd, amsmath, float, graphicx, longtable,color,
  enumerate, fullpage%, bm, stmaryrd
}
\usepackage{amsrefs} %bibliography
\usepackage{mathrsfs}
\usepackage{psfrag}
\usepackage{fouriernc}\DeclareMathAlphabet{\mathcal}{OMS}{cmsy}{m}{n}
\usepackage[all]{xy}
\SelectTips{eu}{} 
\xyoption{curve}
\usepackage{setspace}
\usepackage{relsize, nicefrac}

\usepackage[utf8]{inputenc}    
\usepackage[T1]{fontenc}       
\newcommand{\hide}[1]{}

\selectfont
\usepackage[colorlinks]{hyperref}
\DeclareMathOperator{\add}{add{}}

\DeclareMathOperator{\cok}{cok}

\DeclareMathOperator{\depth}{depth}

\DeclareMathOperator{\End}{End}
\DeclareMathOperator{\Ext}{Ext}

\DeclareMathOperator{\gldim}{gldim}

\DeclareMathOperator{\Hom}{Hom}

\DeclareMathOperator{\Mod}{Mod}
\DeclareMathOperator{\op}{op}
\DeclareMathOperator{\pd}{projdim}%%% or better ``pd'' as before?
\DeclareMathOperator{\id}{injdim}

\DeclareMathOperator{\proj}{proj}

\DeclareMathOperator{\Spec}{Spec}
\DeclareMathOperator{\Supp}{Supp}

\DeclareMathOperator{\Tr}{Tr}

\renewcommand{\phi}{\varphi}

\renewcommand{\mod}{\operatorname{mod}}
\renewcommand{\leq}{\leqslant}

\renewcommand{\geq}{\geqslant}

\renewcommand{\to}{\longrightarrow}

\DeclareMathOperator{\CM}{CM}

\DeclareMathOperator{\Tor}{Tor}

\newcommand{\p}{{\mathfrak{p}}}
\newcommand{\m}{{\mathfrak{m}}}

\newtheorem{theorem}{Theorem}

\newtheorem{prop}[theorem]{Proposition}
\newtheorem{proposition}[theorem]{Proposition}

\newtheorem{lemma}[theorem]{Lemma}

\newtheorem{corollary}[theorem]{Corollary}
\newtheorem{question}[theorem]{Question}

\newtheorem{defn}[theorem]{Definition}
\newtheorem{definition}[theorem]{Definition}

\newtheorem{remark}[theorem]{Remark}
\newtheorem{notation}[theorem]{Notation}

\newtheorem{example}[theorem]{Example}

\numberwithin{theorem}{section}
\numberwithin{equation}{section}

\begin{document}

\setlength\abovedisplayskip{0pt}
\setlength\belowdisplayskip{0pt}

\title{%
Auslander's Theorem and $n$-isolated singularities
}

\author{Josh Stangle}
 %\email{jstangle@uwsuper.edu}

%\thanks{}

\date{\today}

%\keywords{}

\begin{abstract} One of the most stunning results in the representation theory of Cohen-Macaulay rings is Auslander's well known theorem which states a CM local ring of finite CM type can have at most an isolated singularity. There have been some generalizations of this in the direction of countable CM type by Huneke and Leuschke. In this paper, we focus on a different generalization by restricting the class of modules. Here we consider modules which are high syzygies of MCM modules over non-commutative rings, exploiting the fact that noncommutative rings allow for finer homological behavior. We then generalize Auslander's Theorem in the setting of complete Gorenstein local domains by examining path algebras, which preserve finiteness of global dimension.  

%
% ^.

\end{abstract}
\maketitle 

%
% ^.

\section{Introduction}

One main focus of the study of representation theory of commutative Noetherian rings is the question of finite Cohen-Macaulay (CM) type--i.e., when a local commutative Noetherian ring $R$ has only finitely many (up to isomorphism) indecomposable maximal Cohen-Macaulay modules. Auslander showed that a complete Cohen-Macaulay local ring $R$ of finite CM type has at most an isolated singularity; that is, \[\gldim R_\p=\dim R_\p\] for all non-maximal prime ideals $\p\in \Spec R$, \cite{Aus1}. Wiegand \cite{Wiegand:98} and Leuschke-Wiegand \cite{LW:00} then proved that finite CM type ascends to and descends from the completion of an excellent local ring $R$, thus generalizing the theorem to all excellent CM local rings. Finally, Huneke-Leuschke gave a completion-free proof for arbitrary CM local rings in \cite{HL}.
In the paper of Huneke-Leuschke, the idea of countable CM type is addressed, and they are able to show that if a CM local ring has countable CM type then the singular locus is at most one-dimensional. In this paper we are interested in a different generalization of Auslander's theorem. We wish to restrict our finiteness assumption to a smaller class of modules. To do so we will consider non-commutative algebras over which high-syzygies exhibit similar behavior to maximal Cohen-Macaulay modules over CM local rings. Such algebras will be called $n$-canonical orders, and Section \ref{n-canon} will focus on their properties. The main result, Theorem \ref{AuslanderGen}, is that for an $n$-canonical $R$-order $\Lambda$, if there are only finitely many (up to isomorphism) indecomposable $n^{th}$ syzygies of MCM $\Lambda$-modules, then $\gldim \Lambda_\p \leq n+\dim R_\p$ for all non-maximal prime ideals $\p \in \Spec R$. 

Finally, in section \ref{commutative} we refocus on the case of commutative rings, using our main theorem to show that if $R$ is a complete Gorenstein local domain and $Q$ is an acyclic quiver such that $RQ$ has finitely many indecomposable first syzygies (of MCM $RQ$-modules), then $R$ has at most an isolated singularity. This is a generalization of Auslander's theorem for Gorenstein domains.

\section{Background and Notation}

Here we will briefly remind the reader of the notation, conventions, and definitions which are heavily utilized in this article. Throughout, $R$ will be a commutative Noetherian ring of finite Krull dimension $d$. We use the notation $(R,\m,k)$ to imply $R$ is a commutative local Noetherian ring with maximal ideal $\m$ and residue field $R/\m=k$. For the convenience of the reader, below we include definitions of the preliminary notions we will use.

\begin{definition}  Let $(R,\m,k)$ be a commutative local Noetherian ring of finite Krull dimension $d$. 
\mbox{}
\begin{itemize} 

\item Let $M$ be a finitely generated $R$-module. We say $M$ is \emph{maximal Cohen-Macaulay} (MCM) if \[\mathrm{depth} (M):= \mathrm{min} \{i\geq 0\mid \mathrm{Ext}^i_R(k,M)\neq 0\}=d.\]

\item An $R$-algebra $\Lambda$ is an $R$-\emph{order} if it is a MCM $R$-module.

\item Denote by $\Mod \Lambda$ the category of left $\Lambda$-modules and $\mod \Lambda$ the full subcategory of $\Mod \Lambda$ consisting of finitely generated modules. Unless specified otherwise, when we say $M$ is a $\Lambda$-module, we always mean a finitely generated left $\Lambda$-module.

\item We denote by $\CM\Lambda$ the full subcategory of  $\mod \Lambda$ consisting of modules which are maximal Cohen-Macaulay $R$-modules. 

\item For a (possibly non-commutative) ring $\Gamma$, we will denote by $\Gamma^{\op}$ the opposite ring. If $M$ is an abelian group with a right $\Gamma$-module structure, we will say $M\in \mod \Gamma^{\op}$ to indicate that $M$ is a left $\Gamma^{\op}$-module.

\item An $R$-order $\Lambda$ is \emph{non-singular} if $\gldim\Lambda_\p=\dim R_\p$ for all $\p \in \Supp_R\Lambda$.

\item We say an $R$-order $\Lambda$ is an \emph{isolated singularity} if $\gldim\Lambda_\p=\dim R_\p$ for all non-maximal prime ideals $\p\in\Supp_R\Lambda$.

\item For any ring $\Gamma$, we denote by $\proj \Gamma$ the full subcategory  of $\mod \Lambda$ consisting of all projective $\Gamma$-modules.

\item For any module $M$, $\add M$ denotes the \emph{additive closure} of $M$, i.e., the full subcategory of $\Mod \Lambda$ consisting of all modules which are isomorphic to direct summands of finite direct sums of copies of $M$. 

\end{itemize}
\end{definition}

For details on Cohen-Macaulay rings or canonical modules, see \cite{BH}*{Section 3.3}; for details on orders, see \cites{CurtisReiner, Reiner}; for details on module-finite algebras with injective dimension in mind, see \cite{GotoKenjiBass}.

In the case that $R$ is Cohen-Macaulay with a canonical module $\omega_R$, an $R$-order possesses a special module akin to $\omega_R$. 

\begin{definition} Let $R$ be a Cohen-Macaulay ring with canonical module $\omega_R$ and $\Lambda$ an $R$-order. Then the \emph{canonical module} of $\Lambda$ is $\omega_\Lambda=\Hom_R(\Lambda,\omega_R)$.  We see that $\omega_\Lambda$ is both a $\Lambda-$ and $\Lambda^{op}$-module.
\end{definition}

It will be useful to note that the canonical module of an order over a Cohen-Macaulay local ring is invariant under change of base ring.

\begin{lemma}\label{ChangeofBase} Suppose $R$ is a CM local ring with a canonical module $\omega_R$ and that $S\hookrightarrow R$ is  ring extension where $S$ possesses a canonical module $\omega_S$. Suppose $\Lambda$ is an algebra which is both an $S$-order and an $R$-order. Then, the canonical module of $\Lambda$ as an $R$-order is isomorphic (as a $\Lambda$-module) to the canonical module of $\Lambda$ as an $S$-order. 

\end{lemma}

\begin{proof}
All we need to establish is that $\Hom_S(\Lambda,\omega_S)\cong \Hom_R(\Lambda,\omega_R)$, i.e. that the canonical module of $\Lambda$ as an $R$-order agrees with that as an $S$-order.  By \cite{BH}*{Theorem 3.3.7} we see \[\Hom_R(\Lambda, \omega_R) \cong \Hom_R(\Lambda, \Hom_S(R,\omega_S)) \cong \Hom_S(\Lambda \otimes_R R, \omega_S)\cong \Hom_S(\Lambda,\omega_S).\] It is straight-forward to verify this is also an isomorphism of $\Lambda$-modules. \qedhere

\end{proof}

In the rest of this article, $R$ is always assumed to be a Cohen-Macaulay local ring with canonical module $\omega_R$. Let $\Lambda$ be an $R$-order. We have the following functors. 

\begin{itemize}
\item The \emph{canonical dual} $D_d(-):=\Hom_R(-,\omega_R):\CM\Lambda \to \CM \Lambda^{\op}$. Note, this functor is exact on $\CM \Lambda$ since $\Ext_R^i(M,\omega_R)=0$ for $i>0$ and $M$ an MCM $R$-module. 
\item The \emph{Matlis dual} $D:=\Hom_R(-,E)$ where $E$ is the injective hull of the residue field, $k$, of $R$. Letting $\text{f.l.}R$ denote the full subcategory of $\mod R$ consisting of finite length $R$-modules, $D:\text{f.l.}R \to \text{f.l.}R$ is a duality. 
\item The functor $(-)^*:=\Hom_\Lambda(-,\Lambda):\mod \Lambda \to \mod \Lambda^{op}$ which gives a duality $(-)^*:\add \Lambda \to \add \Lambda^{\op}$.
\item The \emph{transpose duality} $\Tr:\underline{\mod} \Lambda \to \underline{\mod} \Lambda^{op}$ given by $\Tr M=\cok f_1^*$, where $P_1\xrightarrow{f_1} P_0 \xrightarrow{f_0} M\ \to 0$ is a minimal projective resolution of $M$.
\item Finally, we denote $\Hom_R(-,R)=(-)^\dagger$. In the case when $R$ is Gorenstein, we note that $D_d(-)=(-)^\dagger$. 
\end{itemize}

\section{Projective Dimension and the Canonical module}

\subsection{$n$-Canonical Orders}

In this section we examine orders which exhibit similar behavior as seen in commutative rings. Specifically, we note that by the Auslander-Buchsbaum formula \cite{AusBuchs}, maximal Cohen-Macaulay modules over commutative rings are either projective or have infinite projective dimension. We prove that for orders over CM rings, finite projective dimension of the canonical modules gives a similar result for high syzygies. The central objects of this paper, $n$-canonical orders, are a generalization of the following well-studied class of orders.

\begin{definition}
 Let $R$ be a Cohen-Macaulay ring with canonical module $\omega_R$ and $\Lambda$ an $R$-order. If $\omega_\Lambda$ is projective as a left $\Lambda$-module, then $\Lambda$ is called a \emph{Gorenstein order}.
\end{definition}

A great deal of work has been done to study Gorenstein orders, see e.g., \cite{CurtisReiner}, \cite{GotoKenjiBass},  \cite{IYAMA_HIRT}, and \cite{IW1}.  These are natural candidates--when $R$ is a CM local ring--for noncommutative crepant resolutions. One reason that Gorenstein orders are so useful is that they exhibit some similar behavior to commutative rings. In particular, they satisfy an Auslander-Buschbaum theorem.  

\begin{lemma}\cite{IW1}*{Lemma 2.16} Let $\Lambda$ be a Gorenstein $R$-order. Then for any $X\in \mod \Lambda$ with $\pd_\Lambda X<\infty$ we have \[\pd_\Lambda X + \depth_R X =\dim R.\]
\end{lemma}

The above result is a special case of the main result of this section, which relates the projective dimension of $\omega_\Lambda$ to the possible projective dimension of all finitely generated $\Lambda$-modules. 

It is important to note that the Gorenstein property is defined to be one-sided. It is natural to ask if the property passes from $\Lambda$ to $\Lambda^{op}$. 

\begin{remark} \label{rmk1}
The Gorenstein property is symmetric, as shown in \cite{IW1}*{Lemma 2.15}: $\Lambda$ is Gorenstein if and only if $\Lambda^{op}$ is Gorenstein. 
\end{remark}

%The following useful result is an application of Theorem \ref{projdimcanon}. The proof is left to the reader.
%\begin{corollary}\label{cor1}
%Let $\Lambda$ be an $R$-order. If $\gldim \Lambda=n+d$, then $\pd_{\Lambda^{\op}} \omega_\Lambda=n$.
%\end{corollary}

The rest of this paper is dedicated to orders with $\pd_{\Lambda^{\op}}\omega_\Lambda\leq n$. As such, we give this condition a name. 

\begin{definition}\label{nCanonDef} Let $R$ be a CM local ring with canonical module $\omega$. Let $\Lambda$ be an $R$-order. We call $\Lambda$ $n$-\emph{canonical} if $\pd_{\Lambda^{\op}} \omega_\Lambda \leq n$.  \end{definition}

The above definition can be considered a one-sided version of the the $n$-Gorenstein property defined in \cite{Enochs}*{Section 9}. There, a left and right Noetherian ring $\Lambda$ is called $n$-\emph{Gorenstein} (alternatively, \emph{Iwanaga-Gorenstein)} if there is a non-negative integer $n$ such that $\id_\Lambda \Lambda \leq n$ and $\id_{\Lambda^{op}}\Lambda \leq n.$ If one defines \emph{left} $n$-\emph{Gorenstein} to mean only $\id_\Lambda \Lambda \leq n$, then the following proposition (an application of \cite{GotoKenji}*{Proposition 1.1(3)}) yields that an $n$-canonical order $\Lambda$ over a $d$-dimensional local Cohen-Macaulay ring $R$ is precisely an $R$-order which is left $(n+d)$-Gorenstein.  

\begin{prop}\label{symmetry}
Let $\Lambda$ be an order over a Cohen-Macualay local ring $R$ of dimension $d$. The following are equivalent:
\begin{enumerate}
\item $\Lambda$ is $n$-canonical, that is, the $\Lambda^{op}$-module $\omega_\Lambda$ has projective dimension at most $n$.
\item The $\Lambda$-module $\Lambda$ has injective dimension at most $d+n$. 
\end{enumerate}
\end{prop}

Notice that the term $n$-Gorenstein has another definition given by Auslander and Reiten in \cite{AuslanderReitenSyzygy}. 
As in Remark \ref{rmk1}, one may ask whether being $n$-canonical is a symmetric property. Remark \ref{rmk1} establishes this in the affirmative when $n=0$. It is further known to be true in the case $n=1$ by tilting theory. In general, this is known as the \emph{Gorenstein symmetry conjecture} and seems quite difficult. The following result follows from the work in \cite{Zaks} and Proposition \ref{symmetry}, and it will be useful for us.

\begin{proposition}\label{prop2} Let $\Lambda$ be an order over a Cohen-Macaulay local ring $R$. If both $\pd_\Lambda\omega_\Lambda$ and $\pd_{\Lambda^{op}}\omega_\Lambda$ are finite, then $\Lambda$ is $n$-canonical if and only if $\Lambda^{op}$ is $n$-canonical. In particular, if $\gldim \Lambda <\infty$, then $\Lambda$ is $n$-canonical if and only if $\Lambda^{op}$ is $n$-canonical.
\end{proposition}

\begin{proof} We observe that
\begin{align*}
\pd_\Lambda\omega_\Lambda& =\max\{i\geq 0 \mid \Ext_\Lambda^i(\omega_\Lambda,\Lambda)\neq 0\} \\ & = \max\{i\geq 0 \mid \Ext^i_{\Lambda^{op}}(\omega_\Lambda,\Lambda)\neq 0\} \\
&=\pd_{\Lambda^{op}}\omega_\Lambda.
\end{align*}
Here, the first and third inequalities follow since we know both $\pd_\Lambda\omega_\Lambda$ and $\pd_{\Lambda^{op}}\omega_\Lambda$ are finite. The second equality follows from the exact duality $D_d:\CM(\Lambda)\to\CM(\Lambda^{op}).$ \qedhere

\end{proof}

\begin{notation} We define an $n^{th}$-\emph{syzygy} of a module $X\in\mod\Lambda$, denoted $\Omega^nX$ to be a module $Y$ appearing in an exact sequence \[0\to Y\to P_{n-1}\to\cdots\to P_0\to X\to 0,\] where the $P_i$ are finitely generated projective $\Lambda$-modules. Denote by $\mathcal{X}_n$ the additive closure of the full subcategory of $\mod \Lambda$ consisting of all $n^{th}$-syzygies of maximal Cohen-Macaulay $\Lambda$-modules, i.e., 

\[\mathcal{X}_n=\add \{\Omega^n\CM\Lambda\}= \add\{ M\in \mod \Lambda \ | \ M\cong \Omega^nX \ \text{for some} \ X\in\CM\Lambda\}.\]
\end{notation} 

Now we move on to prove the main result of this section, a generalized Auslander-Buchsbaum formula for $n$-canonical orders.  It follows from Ext-vanishing imposed by being $n$-canonical.

\begin{lemma} \label{syzext} Suppose $\Lambda$ is an $n$-canonical order over a CM local ring $R$ with a canonical module. If $M\in \CM \Lambda$ then $\Ext_\Lambda^i(M,\Lambda)=0$ for $i>n$. In particular, if $X\in \mathcal{X}_n$, then $\Ext_\Lambda^i(X,\Lambda)=0$ for $i>0$. \end{lemma}

\begin{proof}Begin by taking a projective resolution of $\omega_\Lambda$ over $\Lambda^{\op}$ \[0\to P_n\to P_{n-1} \to \dots \to P_0 \to \omega_\Lambda \to 0\] and apply $D_d(-)$ to get a resolution \[0 \to \Lambda \to I_0 \to \dots \to I_{n-1} \to I_n \to 0\] with $I_j\in \add \omega_\Lambda$. Since $\Ext_\Lambda^i(N,\omega_\Lambda)=0$ for $i> 0$ and $N\in \CM \Lambda$, we have $\Ext_\Lambda^{n+i}(N,\Lambda)=0$. The final statement then follows by dimension shifting for $\Ext$.   \qedhere

\end{proof}

Note that Lemma \ref{syzext} provides a proof of Proposition \ref{symmetry} $(1)\Rightarrow (2)$ since, by the Depth Lemma over $R$, $\Omega^{n+d}(\mod\Lambda)\subseteq \mathcal{X}_n$.

\begin{theorem} \label{projdimcanon} Suppose $\Lambda$ is an $n$-canonical order over a CM local ring with a canonical module. For any $X\in \mod \Lambda$ with $\pd_\Lambda X<\infty$  we have \begin{equation}\dim R \leq \pd_\Lambda X +\depth_R X\leq \dim R +n.\label{nABformula}\end{equation}  
\end{theorem}

\begin{proof} First we show that if $X\in \CM \Lambda$ satisfies $\pd_\Lambda X<\infty$, then $\pd_\Lambda X\leq n$. By Lemma \ref{syzext}, if $X \in \CM \Lambda$, then $\Ext_\Lambda^i(X,\Lambda)=0$ for $i>n$. Since $\Ext_\Lambda^r(X,\Lambda)\not =0$ for $r=\pd_\Lambda X$, we must have either $\pd_\Lambda X\leq n$ or $\pd_\Lambda X=\infty$. 

Now, for any module $X\in \mod \Lambda$ with $\depth_R X= t$, the $(d-t)^{th}$ syzygy must be in $\CM\Lambda$ by the Depth Lemma for $R$ and the fact that $\Lambda$ is MCM over $R$. We then have \[\pd_\Lambda X= (d-t)+\pd\Omega^{d-t}X \leq d-t+n=\dim R-\depth_R X+n.\] The upper bound of (\ref{nABformula}) follows at once from this. To prove the lower bound we simply note that projective $\Lambda$-modules are in $\CM\Lambda$. By the Depth Lemma again, if $\depth_RX=t$, then the first syzygy which could be projective is the $(d-t)^{th}$, as each syzygy can go up in depth by at most 1. Thus \[\pd_\Lambda X \geq d-\depth_R X.\] This concludes the proof.   \end{proof}

It is well known that commutative rings of finite Krull dimension $d$ have either infinite global dimension or finite global dimension equal to $d$. Thus, for $1\leq n< \infty$, we know commutative $n$-canonical orders cannot exist. The goal of the rest of this section is to give some examples of non-commutative $n$-canonical orders for $n\geq 1$.  We begin with a technical lemma which will help in establishing the projective dimension of certain modules. 

\begin{lemma}\label{lemma:1} Let$\Lambda$ be an order over a Cohen-Macualay local ring $(R,\m)$. Let $M$ be a $\Lambda$-module and $x_1,\dots,x_t\in\m$ an $M$-regular sequence. Then we have \[\pd_\Lambda (M/(x_1,\dots,x_t)M)=\pd_\Lambda M+t.\]

\end{lemma}

\begin{proof}
Consider the short exact sequence \[0\to M \xrightarrow{\cdot x_1} M \to M/x_1M \to 0,\]  which induces an exact sequence\[\Ext^{i-1}_{\Lambda}(M,-) \to \Ext^i_{\Lambda}(M /x_1M,-) \to   \Ext^i_{\Lambda}(M,-) \to \Ext_{\Lambda}^i(M, -).\] From this we can deduce $\pd_\Lambda M/x_1M = \pd_\Lambda M +1.$ The result follows from a repeated application of this process. \qedhere
\end{proof}

The next result follows from Lemma \ref{lemma:1} and Theorem \ref{projdimcanon} and is our first example of $n$-canonical orders.

\begin{corollary}\label{cor1} Let $R$ be a $d$-dimensional CM local ring with a canonical module and $\Lambda$ an $R$-order. We have $\gldim\Lambda \geq d$. Moreover, if $\gldim \Lambda <\infty$, then $\gldim\Lambda = \pd_{\Lambda^{op}}\omega_\Lambda$+d.  \end{corollary}

\begin{proof} The first claim follows from Lemma \ref{lemma:1} applied to $M=\Lambda$. For the second claim, we assume that $\gldim\Lambda<\infty$. Since $\Lambda$ is left and right Noetherian, we know $\gldim\Lambda^{op}=\gldim\Lambda$. Then, since $\omega_\Lambda$ is in  $\CM\Lambda^{op}$ we can find an $\omega_\Lambda$-regular sequence $x_1,\dots,x_d\in\m$, where $\m$ is the maximal ideal of $R$. Applying Lemma \ref{lemma:1}, we see that \[\pd_{\Lambda^{op}}\omega_\Lambda + d = \pd_{\Lambda^{op}} (\omega_\Lambda/(x_1,\dots,x_n)\omega_\Lambda).\] Since $\pd_{\Lambda^{op}}(\omega_\Lambda/(x_1,\dots,x_n)\omega_\Lambda)\leq \gldim \Lambda^{op}=\gldim\Lambda,$ we have $\pd_{\Lambda^{op}}\omega_\Lambda \leq \gldim\Lambda - d$. Hence, $\gldim\Lambda\geq\pd_{\Lambda^{op}}\omega_\Lambda+d$. From Theorem \ref{projdimcanon} we can see that $\gldim\Lambda\leq\pd_{\Lambda^{op}}\omega_\Lambda+d$. This establishes the second claim. 

\qedhere

\end{proof}

\begin{remark} Note that the right inequality of Theorem \ref{projdimcanon} cannot be strengthened to an equality. For example, let $\Lambda$ be a finite-dimensional algebra over a field such that $\gldim \Lambda =n>1$ (many such algebras exist). Then, $\Lambda$ is $n$-canonical by Corollary \ref{cor1} and we have \[\{\pd_\Lambda X \mid X\in \CM (\Lambda)\}=\{1,2,\dots,n\}.\]

\end{remark}

\begin{example} Let $k$ be an infinite field and let $R$ be the complete (2,1)-scroll, that is, $R=k[[x,y,z,u,v]]/I$ with $I$ the ideal generated by the $2\times2$ minors of  $\left(\begin{smallmatrix} x& y & u \\ y & z & v\end{smallmatrix}\right)$. Then, $R$ is a 3-dimensional CM normal domain of finite CM type \cite{Yoshino}*{ 16.12}. It is known  $\Gamma=\End_R(R\oplus \omega)$ is MCM over $R$. Moreover, Smith and Quarles have shown $\gldim(\Gamma)=4$ \cite{SQ} while $\dim R=3$. Thus $\Gamma$ is a 1-canonical order.
\label{scroll}
\end{example}

We now establish the existence of $n$-canonical orders for $n\geq 1$ with infinite global dimension. To start, we show that the $n$-canonical property is additive under tensoring. 

\begin{lemma} Let $\Lambda_1$ and $\Lambda_2$ be algebras over a Gorenstein local ring $R$ such that $\Lambda_1$ and $\Lambda_2$ are free $R$-modules. Then $\Lambda_1\otimes_R \Lambda_2$ is an $R$-order and $\omega_{\Lambda_1\otimes\Lambda_2} \cong \omega_{\Lambda_1} \otimes_R \omega_{\Lambda_2}$ as both $\Lambda_1\otimes_R\Lambda_2$-bimodules \label{tensor} \end{lemma}

\begin{proof} This follows from the composition of $\Lambda_1\otimes_R \Lambda_2$-bimodule isomprhisms \begin{align*}
\Hom_R(\Lambda_1\otimes_R\Lambda_2,R)&\cong \Hom_R(\Lambda_1,\Hom_R(\Lambda_2,R)) \\
& \cong \Hom_R(\Lambda_1,R)\otimes_R\Hom_R(\Lambda_2,R).
\end{align*}

Here, the first isomorphism is the canonical isomorphism of Hom-tensor adjunction. The second isomorphism follows since $\Lambda_1$ is a free $R$-module.
\end{proof}

Since we are now able to find the canonical module of orders which are tensor products of free $R$-modules, we get the following examples for $R$ a regular local ring.

\begin{theorem} \label{ncanonmain:1} Let $(R,\m,k)$ be a regular local ring. Suppose $\Lambda_1, \ \Lambda_2$ are $n_1$-canonical and $n_2$-canonical $R$-orders, respectively. Then $\Lambda_1\otimes_R\Lambda_2$ is an $(n_1+n_2)$-canonical $R$-order.\end{theorem}

\begin{proof} Since $\Lambda_1$ and $\Lambda_2$ are MCM over $R$, and $R$ is a regular local ring, then in fact they are free. Then, noting that $(\Lambda_1\otimes_R\Lambda_2)^{\op}\cong\Lambda_1^{\op}\otimes_R\Lambda_2^{\op},$ the theorem follows immediately from Lemma \ref{tensor} and \cite{CE}*{Corollary IX.2.7}, namely that \begin{align*}
\pd_{\Lambda_1^{\op} \otimes_R \Lambda_2^{\op}} (\omega_{\Lambda_1 \otimes_R \Lambda_2} )=&\pd_{\Lambda_1^{\op} \otimes_R \Lambda_2^{\op}} (\omega_{\Lambda_1} \otimes_R \omega_{\Lambda_2})\\
=&\pd_{\Lambda_1^{\op}}\omega_{\Lambda_1}+\pd_{\Lambda_2^{\op}}\omega_{\Lambda_2}. \qedhere \end{align*} \end{proof}

With this in hand, we can prove the existence of orders which are $n$-canonical and have infinite global dimension. We first remind the reader of some basics on path algebras.

\subsection{Homological behavior of Path Algebras}\label{pathalghom}

The main theorem of Chapter 3 is homological in nature. As such, we collect some background on the homological behavior of path algebras.

\begin{definition} A \emph{quiver} $Q=(Q_0,Q_1,s,t)$ is a directed graph $Q$ with vertex set $Q_0$ and arrow set $Q_1.$ There are two maps $s,t:Q_1\to Q_0$ where for an arrow $a\in Q_1,$ $s(a)$ is the origin of $a$ and $t(a)$ is the destination of $a$.  A \emph{path} in $Q$ is a sequence of arrows $a_na_{n-1}\dots a_1$ such that $t(a_i)=s(a_{i+1})$ for $1\leq i \leq n-1.$ For each vertex $a\in Q_0$, we have the \emph{trivial path} at $a$, denoted by $e_a$, which is the path which begins at $a$, ends at $a$, and consists of no arrows.

\end{definition}

\begin{definition} Let $R$ be a commutative Noetherian ring and $Q$ a quiver. The \emph{path algebra} $RQ$ of $Q$ over $R$ is the free $R$-module with the basis the set of all paths $a_la_{l-1}...a_1$ of length $l \geq 0$ in $Q$. The product of two basis vectors (i.e., paths) $b_k...b_1$ and $a_l...a_1$ of $RQ$ is defined  by
\[
(b_k...b_1)\cdot(a_l...a_1)=b_k...b_1a_l...a_1
\]
if $t(a_l)=s(b_1)$ and 0 otherwise, i.e., the product of arrows $b\cdot a$ is nonzero if and only if $b$ leaves the vertex where $a$ arrives. Multiplication is extended to linear combinations of basis elements $R$-linearly. 
\end{definition}

The next result is well-known to experts, and it is what makes path algebras a convenient choice for relating global dimension information about orders back to the commutative base rings; we include a proof for convenience.

\begin{prop}\label{pathalg:1} Let $Q$ be an acyclic quiver. Let $R$ be a regular local ring of dimension $d$ and $RQ$ the path algebra of $Q$ over $R$. Then, $\gldim RQ\leq d+1$. If $R$ is not regular, then $\gldim RQ=\infty$. \end{prop}

\begin{proof} Let $R$ be a regular local ring with $\gldim R = d<\infty$. We note that the path algebra $RQ$ is exactly the tensor algebra $T_S(V)$ where $S:=RQ_0$ is the subring of $RQ$ consisting only of linear combinations of the trivial paths $e_i$ and $V=RQ_1$ is an $S$-bimodule. Then, we have a canonical short exact sequence of $R$-modules \[0\to T_S(V)\otimes_RV\otimes_RT_S(V)\xrightarrow{f}T_S(V)\otimes_R T_S(V) \xrightarrow{mult.} T_S(V)\to 0.\] Here, $f$ is defined on basis elements by $f(x,v,y)=xv\otimes y - x\otimes vy$. To see that $\gldim T_S(V)\leq \gldim R + 1$, we consider a finitely generated left $T_S(V)$-module $M$ which is a $d^{th}$ syzygy over $T_S(V)$. We can find such a module over $T_S(V)$ since, by Corollary \ref{cor1}, we know $\gldim RQ \geq d$.  Such a module is necessarily projective over $R$ since it is MCM over $R$ by the Depth Lemma. It will suffice to show $\pd_{T_S(V)} M \leq 1$. 

The multiplication map $T_S(V)\otimes_R T_S(V) \to T_S(V)$ is split by the map $g:T_S(V)\to T_S(V)\otimes_RT_S(V)$ given by $g(v)=1\otimes v$, so the above exact sequence is split exact as a sequence of right $T_S(V)$-modules, and it remains exact when tensoring  with $M$ over $T_S(V)$. This yields an exact sequence of $T_S(V)$-modules \[0\to T_S(V)\otimes_RV\otimes_R M\to T_S(V)\otimes_R M \to M\to 0.\] We note that $V$ is projective over $R$ by definition; $M$ is projective over $R$ by assumption; and extension of scalars $T_S(V)\otimes_R -$ preserves projectivity. Hence, the above is a projective resolution of $M$ over $T_S(V)$. Thus $\pd_{T_S(V)} M \leq 1$. This verifies $\gldim RQ \leq d+1$.

For the second assertion we observe that $RQ\cong R\otimes_R RQ$, and both $R$ and $RQ$ are free over $R$. Hence by \cite{CE}*{Corollary IX.2.7}, if $R$ has a module of infinite projective dimension (if $R$ is local and non-regular, then the residue field is such a module), then so has $RQ$. \qedhere

\end{proof}

We record one more Lemma in order to work with path algebras efficiently. 

\begin{lemma} \label{pathtensor} Let $R$ be an algebra over a commutative local ring $T$ . Let $Q$ a quiver, and $I$ a right ideal in $TQ$. Then there is an isomorphism of $R$-algebras \[RQ/I(RQ)\cong (TQ/I) \otimes_T R.\]
\end{lemma}

\begin{proof} We begin with the case $I=0$. We regard $TQ$ and $R$ as subrings of $RQ$. Then, the multiplication gives a map $\Phi:TQ\otimes_T R\to RQ$. This map is clearly a morphism of $R$-algebras. Since both sides are free $R$-modules with basis given by all paths on $Q$, this is necessarily an isomprhism. 

We now move to the case that $I$ is a non-zero ideal. We note \[R\otimes_T (TQ/I) \cong R\otimes_T TQ \otimes_{TQ} TQ/I \cong RQ \otimes_{TQ} TQ/I \cong RQ/I(RQ),\] where the second isomorphism follows from the $I=0$ case.  \end{proof}  

And now we can produce a natural example of $n$-canonical orders with infinite global dimension.

\begin{theorem} \label{gorAus} Let $(R,\m,k)$ be a $d$-dimensional Gorenstein local domain. Suppose $Q$ is an acyclic quiver. Then $\Lambda=RQ$ is a 1-canonical $R$-order. If $R$ is not regular, then $\gldim \Lambda=\infty. $
\end{theorem}

For the proof we will need to reduce to the case where $R$ is complete via the following lemma.

\begin{lemma}\label{complete} Suppose $R$ is a CM local ring with a canonical module $\omega_R$ and that $R\hookrightarrow S$ is a faithfully flat (commutative) ring extension such that $\dim S= \dim R$ and $S$ has a canonical module $\omega_S=(\omega_R) \otimes_R S$ (e.g., if $S=\widehat{R}$). Let $\Lambda$ be an $R$-order. We have that $\Lambda$ is an $n$-canonical $R$-order if and only if $\Lambda \otimes_R S$ is an $n$-canonical $S$-order. \end{lemma} 

\begin{proof} We must prove two facts. First we note that since $S$ is faithfully flat \[\Hom_R(M,N)\otimes_R S \cong \Hom_S(M\otimes_R S, N\otimes_R S).\] It follows at once that $\omega_{\Lambda \otimes_R S}\cong (\omega_\Lambda) \otimes_R S$. Verifying that this is a $\Lambda \otimes_R S$-isomorphism is straightforward. Next, since exactness of $\Lambda$-module sequences can be checked as $R$-modules, $S$ is faithfully flat over $R$, and $-\otimes_R S$ takes projective $\Lambda$-modules to projective $\Lambda \otimes_R S$-modules, we see that \[\pd_\Lambda \omega_\Lambda =\pd_{\Lambda \otimes_R S} (\omega_{\Lambda \otimes_R S}).\] The lemma follows at once from these two observations.  \qedhere

\end{proof}

\begin{proof}[Proof of Theorem \ref{gorAus}]  We reduce to the case where $R$ is complete. Let $\widehat{R}$ denote the completion of $R$ with respect to the maximal ideal. By Lemma \ref{complete}, we see that $RQ$ is 1-canonical if and only if $RQ\otimes_R \widehat{R}$ is  1-canonical. But, by Lemma \ref{pathtensor}, we know that $RQ\otimes_R \widehat{R} \cong \widehat{R}Q$. Thus we see $RQ$ is 1-canonical if and only if $\widehat{R}Q$ is 1-canonical. Thus we may assume $R$ is complete.

Now, by Cohen's Structure Theorem for complete local rings,  \cite{Mats}*{Theorem 8.24}, $R$ is an order over some $d$-dimensional regular local ring $S$. Since $R$ is a Gorenstein local ring and an order over $S$, we have $R\cong\omega_R\cong\Hom_S(R,S)$ and $\pd_R\omega_R=0$ since $R$ is MCM over $S$ and hence free; i.e., $R$ is a 0-canonical $S$-order. Now, by Proposition \ref{pathalg:1}, we know that $\gldim SQ\leq d+1$, and hence by Corollary \ref{cor1} $SQ$ is a 1-canonical $S$-order. Now, by Proposition \ref{pathtensor} and Theorem \ref{ncanonmain:1}, $\Lambda:=RQ\cong R\otimes_S SQ$ is a $1$-canonical $S$-order.  We observe that by Lemma \ref{ChangeofBase}, the canonical module of $\Lambda$ as an $S$-order is the same as the canonical module of $\Lambda$ as an $R$-order. Hence, $\Lambda$ is a 1-canonical $R$-order. Lastly, by Lemma \ref{pathalg:1}  we know that if $R$ is not regular, we have $\gldim \Lambda=\infty$. \qedhere

 \end{proof}

\section{Higher Isolated Singularities}\label{n-canon}

The main theorem of this paper is that if an order $\Lambda$ is $n$-canonical and has only finitely many nonisomorphic indecomposable modules in $\Omega^n\CM\Lambda$, then $\Lambda$ has finite global dimension on the punctured spectrum of $R$. In this section we show that over orders with this property, high syzygies behave much like MCM modules over isolated singularities.

\begin{defn} Let $\Lambda$ be an order over a CM ring $R$. We call $\Lambda$ an \emph{$n$-isolated singularity} if \[\gldim \Lambda_\p \leq n+\dim R_\p\] for all non-maximal prime ideals $\p$. We say $\Lambda$ is \emph{$n$-nonsingular} if $\gldim \Lambda_\p\leq n+\dim R_\p$ for all $\p \in \Spec R$.  \end{defn}

\begin{remark} It follows from the definition that if $\Lambda$ is an $n$-isolated singularity, it is also an $m$-isolated singularity for any $m\geq n$. It might be interesting to study ``strict'' $n$-isolated singularities where $\gldim \Lambda_\p = n+\dim R_\p$ for all $\p \in \Spec R$, as well as non-strict ones.
\end{remark}

When $\Lambda$ is a isolated singularity (the $n=0$ case) it is known that all modules in $\CM\Lambda$ are $d^{th}$ syzygies; in fact, the stronger statement that all modules in $\CM \Lambda$ are $d$-torsionfree $\Lambda$-modules is proved in \cite{Auslander4}*{Theorem 7.9}. This fact allows one to bound the projective dimension of modules in $\CM\Lambda$ on the punctured spectrum. Over $n$-isolated singularities, we accomplish this task with the following result.

\begin{lemma} Let $\Lambda$ be an $n$-isolated singularity over a CM local ring $R$. Then if $M\in \text{CM}(\Lambda)$ we have \[\pd M_\p \leq n\] for all non-maximal primes $\p$. \label{pdM}
\end{lemma}

\begin{proof} Let $M\in \CM \Lambda$ and $\p\in \Spec R$. It follows that $M_\p\in \CM\Lambda_\p$. Pick a maximal $M_\p$-regular sequence $x_1, \dots, x_t \in \p R_\p$.  It follows from Lemma \ref{lemma:1} that  \[ \pd_{\Lambda_\p} M_\p/(x_1,\dots,x_t) M_\p=\pd_{\Lambda_\p} M_\p +\dim R_\p.\] Since $\gldim \Lambda_\p \leq n+\dim R_\p$, it must be that \[\pd_{\Lambda_\p} M_\p\leq n.\qedhere\]  

\end{proof}

From this we get the following useful characterization of $n$-isolated singularities. Recall that $\mathcal{X}_n=\add\{\Omega^n\CM\Lambda\}$.

\begin{lemma}\label{syzfreepunc} Let $\Lambda$ be an order over a CM local ring $R$. Then $\Lambda$ is an $n$-isolated singularity if and only if $X_\p$ is a projective $\Lambda_\p$-module for all $X\in\mathcal{X}_n$ and non-maximal primes $\p\in \Spec R$.  \end{lemma}

\begin{proof} ($\Rightarrow$): This follows at once from the previous lemma.

($\Leftarrow$): Fix a non-maximal prime $\p \in \Spec R$. We must show $\gldim \Lambda_\p \leq n+\dim R_\p$. Let $X\in \mod\Lambda_\p$. We know, then, that $X=M_\p$ for some $M\in \mod \Lambda$. We note that $\Omega^{n+d}_\Lambda M\in \mathcal{X}_n$, so $\Omega^{n+d}_{\Lambda_\p}X=(\Omega^{n+d}_{\Lambda}M)_\p$ is a projective $\Lambda_\p$-module. Hence, $\pd_{\Lambda_\p}X< \infty$. Thus, we have established that $\gldim \Lambda_\p <\infty$. Since $\omega_\Lambda \in \CM \Lambda,$ it is clear that $\Omega_\Lambda^n(\omega_\Lambda)\in\mathcal{X}_n$. Hence $\Omega_{\Lambda_\p}^n(\omega_{\Lambda_\p})=(\Omega_\Lambda^n(\omega_\Lambda))_\p$ is projective by assumption. Thus, $\pd_{\Lambda_\p}\omega_{\Lambda_\p}\leq n$. Applying Corollary \ref{cor1} to $\Lambda_\p^{op}$, we obtain $\gldim \Lambda_\p = \gldim \Lambda_\p^{op}\leq n+\dim R_\p.$
\end{proof}

The following lemma will be useful later, as it detects $n$-isolated singularities.

\begin{lemma} \label{ext1} Let $R$ be a CM local ring with canonical module $\omega$. Let $\Lambda$ be an $R$-order. Then $\Lambda$ is an $n$-isolated singularity if and only if $\ell_R(\Ext_\Lambda^1(N,M))<\infty$ for all $M,N \in \mathcal{X}_n$. \end{lemma}

\begin{proof}  ($\Rightarrow$): This follows at once from Lemma \ref{pdM}.  

($\Leftarrow$): Suppose $\ell(\Ext_\Lambda^1(N,M))< \infty$ for all $M,N\in \mathcal{X}_n$. Let $\p$ be a prime ideal of $R$ which is not maximal. Consider a module $X\in \mod \Lambda_\p$.  We know $X=M_\p$ for some $M\in\mod\Lambda$. Consider the exact sequence over $\Lambda_\p,$ \begin{equation} 0\to \Omega_{\Lambda_\p}^{n+d+1} (X) \to F \to \Omega_{\Lambda_\p}^{n+d}(X) \to 0\label{exact1},\end{equation} where $F$ is a free $\Lambda_\p$-module. This is the same as the exact sequence \begin{equation} 0\to \Omega_{\Lambda_\p}^{n+d+1} (M_\p) \to F \to \Omega_{\Lambda_\p}^{n+d}(M_\p) \to 0\label{exact2}.\end{equation}  Since $\Omega_{\Lambda_\p}^i(X)=(\Omega_{\Lambda}^i M)_\p$ for all $i\geq 0$ and $\Omega^{n+d}M,\Omega^{n+d+1}M 
\in \mathcal{X}_n$, it follows that \[\Ext_{\Lambda_\p}^1(\Omega_{\Lambda_\p}^{n+d}X,\Omega_{\Lambda_\p}^{n+d+1} X)\cong\left(\Ext_{\Lambda}^1(\Omega^{n+d}M,\Omega^{n+d+1}M)\right)_\p=0.\] Where the final equality follows since $\Ext^1_\Lambda(\Omega^{n+d}M,\Omega^{n+d+1}M)$ has finite length by assumption. This means sequence (\ref{exact1}) splits, and hence $\Omega_{\Lambda_\p}^{n+d}X$ is $\Lambda_\p$-projective, and, therefore, $\pd_{\Lambda_\p}X<\infty$ and so $\gldim \Lambda_\p<\infty$.  We can apply a similar argument as above to $X=\omega_{\Lambda_\p}=(\omega_\Lambda)_\p$ to see that $\pd_{\Lambda_p}\omega_{\Lambda_p}\leq n$. In this case, we only need to consider $\Omega^nX$ since we already have $\omega_\Lambda\in\CM\Lambda$. Thus we conclude $\Lambda_p^{op}$ is $n$-canonical. An identical argument to the end of the proof of Lemma \ref{syzfreepunc} gives that $\gldim \Lambda_\p\leq n+\dim R_\p$.  \qedhere

\end{proof}

The next proposition illustrates that $n^{th}$syzygies (of MCM  modules) over an $n$-isolated singularity behave like MCM modules over an isolated singularity. This is shown for the $n=0$ case in \cite{IYAMA_HIRT}*{Theorem 1.3.1}, the included proof is largely the same.

\begin{prop} Let $\Lambda$ be an $n$-isolated singularity over a $d$-dimensional CM local ring $R$. For $X\in \mathcal{X}_n$:

\begin{enumerate}
\item $\Ext^i_{\Lambda^{op}}(\Tr X, \Lambda)=0$ for $i=1,\dots,d$.
\item $\Ext_\Lambda^i(X,Y)$, $\Tor_i^\Lambda(Z,X)$, and $\underline{\Hom}_\Lambda(X,Y)$ are all finite length for any $Y\in \mod\Lambda$ and $Z\in \mod \Lambda^{\op}$.  
\end{enumerate}
\label{whysyzygies}
\end{prop}

\begin{proof} Let $\p \in \Spec R$ be non-maximal. We see if $X \in\mathcal{X}_n$, then $X_\p$ is projective over $\Lambda_\p$ by Lemma \ref{pdM}, and thus (2) holds. For assertion (1), we note that if $d=0$ there is nothing to show. In the case where $d=1$, $X$ is projective on the punctured spectrum, by Lemma \ref{syzfreepunc}. This implies that $\Ext_{\Lambda^{op}}^1(\Tr X, \Lambda)$ has finite length since $\Tr X_\p=0$ for any non-maximal prime ideal $\p$. Then the well-known exact sequence (see, e.g., \cite{LW}*{Proposition 12.8}) \[0 \to \Ext_{\Lambda^{op}}^1(\Tr X, \Lambda) \to X \to X^{**} \to \Ext_{\Lambda^{op}}^2( \Tr X, \Lambda) \to 0 \] shows that $\Ext_{\Lambda^{\op}}^1(\Tr X, \Lambda)$ embeds in $X$. But, $\depth_R X \geq 1$ since $d\geq 1$ so $X$ cannot contain a module of depth zero. Thus, $\Ext_{\Lambda^{op}}^1(\Tr X, \Lambda)=0$.

Now suppose $\Ext_{\Lambda^{op}}^i(\Tr X,\Lambda)=0$ for $i=1,\dots,k-1$ for some $2\leq k \leq d$. We begin with a projective resolution \[ ...\to P_k \to P_{k-1} \to \dots \to P_1 \to P_0 \to \Tr X \to 0.\] Dualizing the above exact sequence, and utilizing the fact that $X\cong X^{**}$, we get an exact sequence \[0 \to X \to P_2^* \to P_3^* \to \dots \to P_{k-1}^* \to (\Omega^k \Tr X)^* \to \Ext_{\Lambda^{op}}^{k}(\Tr X, \Lambda) \to 0,\]  where $\depth_R (\Omega^k \Tr X)^* \geq 2$. Since $\Ext^{k}_{\Lambda^{op}}(\Tr X, \Lambda)$ has finite length and $P_i^*\in \CM \Lambda$ for all $i$, the Depth Lemma implies $\depth_R X\leq d-1$, which is impossible since $X\in \CM\Lambda$. Thus, it must be that $\Ext_{\Lambda^{op}}^k(\Tr X,\Lambda)=0$. Thus part (1) is proved by induction.     \end{proof}

The following is the analog of \cite{IW1}*{Prop 2.17}, and the proof is similar.

\begin{prop} Let $\Lambda$ be an order over a CM ring $R$ of Krull dimension $d$ with canonical module $\omega_R$. The following are equivalent:

\begin{enumerate} 
\item $\Lambda$ is $n$-nonsingular.
\item $\gldim \Lambda_\m \leq n+d$ for all maximal ideals $\m \in \Spec R$. 
\item $\CM\Lambda \subset \pd_{\leq n} \Lambda$.
\item $\pd_{\Lambda^{\op}} \omega_\Lambda \leq n$ and $\gldim \Lambda <\infty$.
\end{enumerate}

\label{nnon_equiv}
\end{prop}

\begin{proof} The first 3 implications are the same argument as \cite{IW1}, but we include them for the convenience of the reader.
$(1)\Rightarrow (2)$  This is immediate. 

$(2) \Rightarrow (3)$ This proof is nearly identical to the proof of Lemma \ref{pdM}.

$(3) \Rightarrow (4)$ Since $\omega_\Lambda \in \CM \Lambda$, we know it has projective dimension over $\Lambda$ at most $n$ by (3). Also, since each $d^{th}$ syzygy is MCM by the Depth Lemma, we have $\gldim \Lambda< \infty$. Since $\gldim\Lambda^{\op}=\gldim\Lambda<\infty$, Proposition \ref{prop2} then implies that $\pd_{\Lambda^{op}}\omega_\Lambda\leq n.$  

$(4) \Rightarrow (1)$ Let $X$ be in $\CM(\Lambda_\p)$. Since localization can only reduce projective dimension, we have that $\pd_{\Lambda^{\op}_\p} \omega_{\Lambda_\p} \leq n$ and $\gldim \Lambda_\p<\infty$.  The result then follows from Corollary \ref{cor1}. \end{proof}

\begin{remark}\label{Rmk:48}  One might ask if we can strengthen condition (3) to be a set equality. If $n\geq 1$, the answer is no: consider a regular sequence  $\underline{x}=x_1,\dots,x_d$ on $\Lambda$, and take the Koszul complex over $\Lambda$ on $\underline{x}$. Then this is exact and has length $d$. Then $\Omega^{d-1}(\Lambda/\underline{x}\Lambda)$ has depth $d-1$ by the Depth Lemma, but the end of the Koszul complex gives a length one resolution. Thus $\Omega^{d-1}(\Lambda/\underline{x}\Lambda) \in \pd_{\leq n} \Lambda$ but is not in $\CM \Lambda$. It is clear that $(3)$ is equivalent to $\mathcal{X}_n \subset \proj \Lambda$. \end{remark}

\section{Gorenstein Projectives and Auslander's Theorem}

The goal of this section is to prove the following variation of Auslander's Theorem, \cite{Aus1}.

\begin{theorem} \label{AuslanderGen} Let $R$ be a CM local ring with canonical module and suppose $\Lambda$ is an $R$-order such that $\Lambda$ and $\Lambda^{op}$ are $n$-canonical. If $\Lambda$ has only finitely many nonisomorphic indecomposable modules in $\mathcal{X}_n$, then $\Lambda$ is an $n$-isolated singularity. 
\end{theorem}

The proof of this will rely on the notion of \emph{Gorenstein Projective} modules. Originally defined by Auslander and Bridger in \cite{AB}, a module $M$ over an order $\Lambda$ is called \emph{Gorenstein Projective} if $M$ is reflexive (i.e., the natural map $M\to M^{**}$ is an isomorphism) and \[\Ext_\Lambda^i(M,\Lambda)=\Ext_{\Lambda^{op}}^i(M^*,\Lambda)=0\] for all $i>0$.  

We let $G\proj \Lambda$ denote the full subcategory of $\mod \Lambda$ consisting of all Gorenstein projective modules. Our interest in Gorenstein projectives is motivated by the following fact, which is a variation of a well-known property for $n$-Gorenstein rings. For details, see \cite{Enochs}*{Section 10.2}.

\begin{prop} Let $R$ be a CM local ring with canonical module $\omega$. Suppose $\Lambda$ is an $R$-order such that $\Lambda$ and $\Lambda^{op}$ are $n$-canonical, where $n\geq 1$. Let $M$ be a $\Lambda$-module. Then, $G\proj \Lambda=\mathcal{X}_n$ holds. \end{prop}

\begin{proof} Since Gorenstein projectives occur as syzygies in complete resolutions, it is clear that \begin{equation*}G\proj\Lambda\subset \add \Omega^n\CM\Lambda.\end{equation*} To complete the proof we show the reverse inclusion. Let $M=\Omega^n X$ for a maximal Cohen-Macaulay module $X$, and suppose $M$ is not a projective module. By Lemma \ref{syzext} we have that $\Ext_\Lambda^i(M,\Lambda)=0$ for all $i>0$. Then, by dualizing a projective resolution of $M$, we get an exact sequence $$0\to M^* \to P_0^* \to P_1^* \to \cdots.$$ Therefore, we see $M^*$ is an arbitrarily high syzygy. By Lemma \ref{syzext} again we have $\Ext^i_{\Lambda^{op}}(M^*,\Lambda)=0$ for $i>0$. All that remains to show is that $M$ is reflexive. Note that $\Tr M$ fits into the above exact sequence as follows \[0\to \Tr M \to P_2^* \to P_3^* \to \cdots.\] Thus, $\Tr M$ is also an arbitrarily high syzygy and satisfies the same $\Ext$ vanishing as $M$. Thus the exact sequence  \[0\to \Ext_{\Lambda^{op}}^1(\Tr M, \Lambda) \to M \to M^{**} \to \Ext_{\Lambda^{op}}^2(\Tr M, \Lambda) \to 0\] implies that $M\cong M^{**}$.   \qedhere

\end{proof}

The key use of Gorenstein projectives is that they are closed under extensions. This has been shown in various places, see e.g., \cite{AR91}*{Proposition 5.1}.

\begin{corollary} \label{extclosed} Let $\Lambda$ be an order over a CM local ring $R$ with a canonical module. Suppose $\Lambda$ and $\Lambda^{op}$ are $n$-canonical. Then $\mathcal{X}_n$ is closed under extensions.
\end{corollary}

We now return to proving the main theorem. The proof of this involves several lemmas, and it follows closely Huneke and Leuschke's proof of Auslander's Theorem, \cite{HL}. The following Theorem due to Miyata is our first step.

\begin{lemma}\cite{Miyata}*{Theorem 2} Let $\Lambda$ be a module finite algebra over a commutative Noetherian ring $R$. Suppose we have an exact sequence of finitely generated $\Lambda$-modules \[M\to X \to N \to 0\] and that $X\cong M\oplus N$. Then the sequence is a split short exact sequence. \end{lemma}

From this we are able to prove the following lemma about $\Ext^1_\Lambda(N,M)$. The proof is similar to  the one in \cite{HL}; it is omitted for this reason. 

\begin{lemma} \label{MiyataLemm} Let $(R,\m)$ be a CM local ring and $\Lambda$ an $R$-order. Fix $r \in\m$. Suppose we have an exact sequence of $\Lambda$-modules, \[\alpha : \ 0\to M \to X_\alpha \to N \to 0\] and a commutative diagram 

\bigskip
\begin{center}
$\begin{CD} 
\alpha: \ 0 @>>> M @>>> X_\alpha @>>> N @>>> 0 \\
			@. @VrVV      @VfVV  @| @.  \\
r\alpha: \ 0  @>>> M @>>>X_{r\alpha} @>>> N @>>> 0.
\end{CD}$
\end{center}

\medskip

\noindent If $X_\alpha \cong X_{r\alpha}$, then $\alpha \in r\Ext_\Lambda^1(N,M)$.

\end{lemma}

Now, we are able to prove the following lemma from which the main theorem follows. The proof is a straight-forward generalization of the commutative case.

\begin{lemma} \label{HLlemma} Suppose $\Lambda$ is an order over a CM local ring $(R,\m,k)$.  Given $\Lambda$-modules $M$ and $N$, if there are only finitely many choices (up to isomorphism) for $X$ such that there is an exact sequence of $\Lambda$-modules \[0\to M \to X \to N \to 0,\] then $\Ext^i_\Lambda(N,M)$ is a finite length $R$-module. \end{lemma}

\begin{proof} Let $\alpha \in \Ext_\Lambda^1(N,M)$ and $r\in \m$. It is well known that an $R$-module $M$ has finite length if and only if for all $r\in \m$ and $x\in M$ there is an integer $n$ so that $r^nx=0$. Thus, we must only show that $r^n\alpha=0$ for $n\gg 0$. For any integer $n$ we consider a representative \[r^n \alpha: \ \ \ \ 0\to M \to X_n \to N \to 0.\] Since only finitely many $X_n$ can exist up to isomorphism there is an infinite sequence $n_1<n_2<n_3<\dots$ such that $X_{n_i} \cong X_{n_j}$ for all pairs $i,j$. Set $\beta=r^{n_1}\alpha$, and let $i>1$. Then $r^{n_i}\alpha=r^{n_i-n_1}\beta$. We show $\beta=0$. We have, for each $i$, a commutative diagram 
\bigskip

\begin{center}
$\begin{CD} 
\hspace{6ex} \beta: \ 0 @>>> M @>>> X_{n_1} @>>> N @>>> 0 \\
			@.  @Vr^{n_i-n_1}VV      @VVV  @|| @. \\
r^{n_i-n_1}\beta: \ 0  @>>> M @>>>X_{n_i} @>>> N @>>> 0.
\end{CD}$
\end{center} 
\bigskip
\noindent By Lemma \ref{MiyataLemm}, since $X_{n_1} \cong X_{n_i}$, we have $\beta \in r^{n_i-n_1}\Ext_\Lambda^1(N,M)$ for every $i$. Since the sequence of $n_i$ is infinite and strictly increasing, this means $\beta\in \m^t\Ext_\Lambda^1(N,M)$ for all $t$. Finally, the Krull Intersection Theorem \cite{Mats}*{Theorem 8.10} implies $\beta=0$.\qedhere  

\end{proof}

Finally, we prove Theorem \ref{AuslanderGen}. Let us recall $\mathcal{X}_n=\add\{\Omega^n\CM\Lambda\}$. 

\begin{proof}[Proof of Theorem \ref{AuslanderGen}] Let $M,N\in \mathcal{X}_n.$ By Lemma \ref{ext1} we must only show that $\ell_R(\Ext_\Lambda^1(N,M))<\infty$. Consider any sequence $\alpha \in \Ext_\Lambda^1(N,M),$ \[\alpha: \ 0\to M \to X \to N \to 0.\] By Corollary \ref{extclosed}, we know $X\in \mathcal{X}_n$. Now since $M$ and $N$ are finitely generated and there are only finitely many indecomposable modules in $\mathcal{X}_n$, there are only finitely many possibilities for $X$. Namely, $X$ must be one of the finitely many modules in $\mathcal{X}_n$ generated by at most $\mu_\Lambda(M)+\mu_\Lambda(N),$ where $\mu_\Lambda(Y)$ denotes the minimum number of generators of $Y$ over $\Lambda$. Thus, $\ell_R(\Ext_\Lambda^1(N,M))<\infty$ by Lemma \ref{HLlemma}.   \qedhere

\end{proof}

\section{Application to Commutative Rings} \label{commutative}

In view of Theorem \ref{AuslanderGen} and \ref{gorAus}, we arrive at the following generalization of Auslander's Theorem in the case where $R$ is a suitable Gorenstein local ring.

\begin{corollary}\label{AuslanderPath} Let $R$ be a Gorenstein local ring which is an order over a regular local ring $S$ (e.g., if $R$ is complete), and let $Q$ be an acyclic quiver. If there exist only finitely many nonisomorphic indecomposable modules in $\Omega\CM(RQ)$, then $R$ is an isolated singularity, i.e., \[\gldim R_\p=\dim R_\p\] for all non-maximal primes ideals $\p \in \Spec R$. \end{corollary}

\begin{proof} We notice that by Theorem \ref{gorAus} $RQ$ is a 1-canonical order. Additionally, it is clear that $RQ^{op}$ is obtained by taking the path algebra over $R$ of the quiver obtained by reversing all arrows in $Q$. It follows that $RQ^{op}$ is also a 1-canonical $R$-order.  Thus, by Theorem \ref{AuslanderGen}, if there are only finitely many indecomposable modules in $\Omega\CM(RQ)$ we must have that $RQ$ is a 1-isolated singularity. By Proposition \ref{pathalg:1}, we know for any commutative ring $R$ that $\gldim RQ<\infty$ if and only if $\gldim R<\infty$. Thus,  $RQ$ can be a 1-isolated singularity if and only if $\gldim R_\p <\infty$ for all non-maximal primes $\p$. Since $R$ is commutative, this is only possible if $\gldim R_\p=\dim R_\p $. \qedhere

\end{proof}

We note that the proof of Corollary \ref{AuslanderPath} does not require completeness beyond ensuring $R$ is an order over a regular local ring. It would be nice to remove this assumption. In this vein we have the following question:

\begin{question} For a local ring $R$ and an acyclic quiver $Q$, is it true that $RQ$ has only finitely many indecomposable modules in $\mathcal{X}_n$ if and only if $\widehat{RQ}\cong\widehat{R}Q$ has?\end{question}

This question does not appear to be a straightforward generalization of the techniques used by Wiegand in \cite{Wiegand:98}. Even for path algebras, ascent from $RQ$ to the henzelisation $R^hQ$ seems to be difficult.  

\section*{Acknowledgements} The author would like to thank his advisor for much guidance while exploring the mathematics in this article. Through his advisor, this work was partially supported by the National Security Agency.

The author would also like to extend hearty thanks to the referee who made many excellent suggestions and had some very keen insights. The paper is substantially better thanks to their dedication and diligence.

%\section{Acknowledgements}

%\bibliographystyle{amsalpha}
%\bibliography{Refs-2015}

\linespread{1}\selectfont
%\printbibliography

% \newcommand{\arxiv}[2][AC]{\mbox{\href{http://arxiv.org/abs/#2}{\textsf{arXiv:#2}}}}
% \newcommand{\oldarxiv}[2][AC]{\mbox{\href{http://arxiv.org/abs/math/#2}{\textsf{arXiv:math/#2[math.#1]}}}}
% \providecommand{\MR}[1]{\mbox{\href{http://www.ams.org/mathscinet-getitem?mr=#1}{#1}}}
% \renewcommand{\MR}[1]{%
%   {\href{http://www.ams.org/mathscinet-getitem?mr=#1}{MR #1}.}}
% \providecommand{\bysame}{\leavevmode\hbox to3em{\hrulefill}\thinspace}
% \newcommand{\arXiv}[1]{%
%   \relax\ifhmode\unskip\space\fi\href{http://arxiv.org/abs/#1}{arXiv:#1}}

% \def\cprime{$'$}
% \providecommand{\bysame}{\leavevmode\hbox to3em{\hrulefill}\thinspace}
% \providecommand{\MR}{\relax\ifhmode\unskip\space\fi MR }
% % \MRhref is called by the amsart/book/proc definition of \MR.
% \providecommand{\MRhref}[2]{%
%   \href{http://www.ams.org/mathscinet-getitem?mr=#1}{#2}
% }
% \providecommand{\href}[2]{#2}

\newpage
\bibliographystyle{plain}
\bibliography{biblio}

\end{document}